\documentclass[12pt]{article}

\usepackage[dvipdfm, bookmarksopen=true]{hyperref}
\usepackage[centertags]{amsmath}
\usepackage{amsfonts}
\usepackage{amssymb}
\usepackage{amsthm}
\usepackage{pstricks}
\usepackage{cite}

\theoremstyle{plain}
\newtheorem{thm}{Theorem}[section]
\newtheorem{cor}[thm]{Corollary}
\newtheorem{lem}[thm]{Lemma}
\newtheorem{prop}[thm]{Proposition}

\theoremstyle{definition}

\newtheorem{defn}[thm]{Definition}
\newtheorem{exam}[thm]{Example}

\theoremstyle{remark}
\newcommand{\beast}{\begin{eqnarray*}}
\newcommand{\eeast}{\end{eqnarray*}}

\title{{\bf Correction of a theorem on the symmetric group generated by transvections}
}

\author{Hau-wen Huang}

\date{}

\begin{document}
\maketitle

\begin{abstract}
Let $V$ denote a vector space over two-element field $\mathbb F_2$ with finite positive dimension and endowed with a symplectic form $B.$ Let ${\rm SL}(V)$ denote the special linear group of $V.$ Let $S$ denote a subset of $V.$ Define $Tv(S)$ as the subgroup of ${\rm SL}(V)$ generated by the transvections with direction $\alpha$ for all $\alpha\in S.$ Define $G(S)$ as the graph whose vertex set is $S$ and where $\alpha,\beta\in S$ are connected whenever $B(\alpha,\beta)=1.$ A well-known theorem states that under the assumption that $S$ spans $V,$ the following (i), (ii) are equivalent:
\begin{enumerate}
\item[(i)] $Tv(S)$ is isomorphic to a symmetric group.

\item[(ii)] $G(S)$ is a claw-free block graph.
\end{enumerate}
We give an example which shows that this theorem is not true. We give a modification of this theorem as follows. Assume that $S$ is a linearly independent set of $V$ and no element of $S$ is in the radical of $V.$ Then the above (i), (ii) are equivalent.
\end{abstract}



\section{A theorem on the symmetric group generated by transvections}
Throughout this note  let $V$ denote a vector space over two-element field $\mathbb F_2$ with finite positive dimension and endowed with a symplectic form $B.$ Let ${\rm rad}\hspace{0.5mm} V$ denote the radical of $V$ with respect to $B.$ Let ${\rm SL}(V)$ denote the
special linear group of $V.$ For $\alpha\in V$ define a linear
transformation $\tau_\alpha:V\to V$ by
$$
\tau_\alpha\beta=\beta+B(\beta,\alpha)\alpha \qquad \quad
\hbox{for all $\beta\in V.$}
$$
We call $\tau_\alpha$ the {\it transvection on $V$ with direction
$\alpha.$} Observe that $\tau_\alpha^2=1$ and so $\tau_\alpha\in {\rm SL}(V).$
For a subset $S$ of $V$
define $Tv(S)$ to be the subgroup of ${\rm SL}(V)$ generated by
the transvections $\tau_\alpha$ for all $\alpha\in S,$ and define
$G(S)$ to be the simple graph which has vertex set $S$ and an edge
between vertices $\alpha$ and $\beta$ if and only if
$B(\alpha,\beta)=1.$

Let $G$ denote a simple graph. A {\it cut-vertex} of $G$ is a vertex whose deletion increases the number of components. A {\it block} of $G$ is a maximal connected subgraph of $G$ that has no cut-vertex. A {\it block graph} is a simple connected graph in which every block is a complete graph. A {\it claw} is a tree with one internal vertex and three leaves. A simple graph is said to be {\it claw-free} if it does not contain a claw as an induced subgraph.

Let $S$ denote a subset of $V.$ Let $C$ denote the set consisting of all cut-vertices of $G(S).$ We now view $G(S)$ as a 1-dimensional complex. Let $G_1,G_2,\ldots,G_k$ denote the components of $G(S)\setminus C.$ For each $1\leq i\leq k$ define $G_i^\ast$ to be the closure of $G_i$ in $G(S)$. Let $H$ denote the graph with $G_1^\ast,G_2^\ast,\ldots, G_k^\ast$ as vertices and an edge between $G_i^\ast$ and $G_j^\ast$ if $G_i^\ast\cap G_j^\ast$ is nonempty.  \cite[Theorem~3.1]{huris:85} states that under the assumption that $S$ spans $V,$ the group $Tv(S)$ is isomorphic to a symmetric group if and only if the following
  (i)--(iv) hold:
\begin{enumerate}
\item[(i)] $G(S)$ is connected;

\item[(ii)] for each $\alpha\in S$ the graph $G(S\setminus\{\alpha\})$ contains at most two components;

\item[(iii)] for each $1\leq i\leq k$ the graph $G_i^\ast$ is a complete graph;

\item[(iv)] $H$ is a tree.
\end{enumerate}
We remark that $G_1^\ast,G_2^\ast,\ldots,G_k^\ast$ are the blocks of $G(S).$ Therefore $G(S)$ is a claw-free block graph if and only if (i)--(iii) hold. Condition (ii) implies that $H$ is acyclic and therefore (i), (ii) imply (iv).  We can state \cite[Theorem 3.1]{huris:85} as follows.

\begin{thm}\label{thm1.1}{\rm \cite[Theorem 3.1]{huris:85}.}
Assume that $S$ spans $V.$ Then $Tv(S)$ is isomorphic to a symmetric group if and only if $G(S)$ is a claw-free block graph.
\end{thm}

\section{A counterexample to the necessity of Theorem~\ref{thm1.1}}
In this section we show a counterexample to the necessity of Theorem~\ref{thm1.1}. We begin by recalling some background material from \cite{brohum:86-1,brohum:86-2}.

\begin{defn}\label{defn2.1}
\cite[Section 3]{brohum:86-1}.
Define a binary relation $\mathcal{T}_0$ on the power set of $V$ as follows. For any two
$S,S'\subseteq V$ we say that $S$ is {\it $\mathcal T_0$-related to} $S'$ whenever there exist $\alpha,\beta\in S$ such that $S'$ is obtained from $S$ by changing $\beta$ to $\tau_\alpha \beta.$
\end{defn}


\begin{defn}\label{defn2.2}
\cite[Section 3]{brohum:86-1}.
Define $\mathcal T$ to be the equivalence relation on the power set of $V$ generated by $\mathcal{T}_0.$
\end{defn}

\begin{lem}\label{lem2.1}
{\rm\cite[Corollary 11.2]{brohum:86-2}.}
Assume that $S$ is a subset of $V$ and no element of $S$ is in ${\rm rad}\hspace{0.5mm}V.$ Let the equivalence relation $\mathcal T$ be as in Definition~\ref{defn2.2}. Then $Tv(S)$ is isomorphic to a symmetric group if and only if there exists $S'$ in the $\mathcal T$-equivalence class of $S$ for which $G(S')$ is a path.
\end{lem}


\begin{exam}
Let $V$ denote a vector space over $\mathbb{F}_2$ with dimension $n\geq 3.$ Let $I=\{\alpha_1,\alpha_2,\ldots,\alpha_n\}$ denote a basis of $V.$ Define a symplectic form $B:V\times V\to \mathbb{F}_2$ by
\begin{equation*}
B(\alpha_i,\alpha_j)=\left\{
\begin{array}{ll}
1 \qquad &\hbox{if $|i-j|=1$,}\\
0 \qquad &\hbox{if $|i-j|\neq 1$}
\end{array}
\right.
\qquad \quad (1\leq i,j\leq n).
\end{equation*}
Let $S=I\cup\{\alpha_1+\alpha_2\},$ which spans $V.$ The set $I$ can be obtained from $S$ by changing $\alpha_1+\alpha_2$ to $\tau_{\alpha_1}(\alpha_1+\alpha_2)=\alpha_2.$ Therefore $S$ is $\mathcal T_0$-related to $I$. The graph $G(I)$ is a path.
By Lemma~\ref{lem2.1} the group $Tv(S)$ is isomorphic to a symmetric group. We draw $G(S)$ as follows.

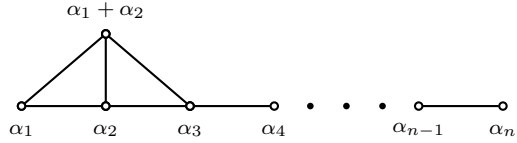
\psset{unit=1.6cm}
\begin{pspicture}(0,0)(9,1.6)
\psset{linewidth=0.8pt}

\psline{o-o}(2.8,0.4)(3.5,1)
\psline{o-o}(3.5,1)(4.2,0.4)

\psline{o-o}(3.5,1)(3.5,0.4)

\psline{o-o}(2.8,0.4)(3.5,0.4)
\psline{o-o}(3.5,0.4)(4.2,0.4)
\psline{o-o}(4.2,0.4)(4.9,0.4)
\psline{o-o}(6.1,0.4)(6.8,0.4)

\rput(3.5,1.2){{\scriptsize $\alpha_1+\alpha_2$}}
\rput(2.8,0.2){{\scriptsize $\alpha_1$}}
\rput(3.5,0.2){{\scriptsize $\alpha_2$}}
\rput(4.2,0.2){{\scriptsize $\alpha_3$}}
\rput(4.9,0.2){{\scriptsize $\alpha_4$}}
\rput(6.1,0.2){{\scriptsize $\alpha_{n-1}$}}
\rput(6.8,0.2){{\scriptsize $\alpha_n$}}

\psdots [dotsize=2.5pt]
(5.2,0.4)(5.5,0.4)(5.8,0.4)

\end{pspicture}

\noindent The block of $G(S)$ with vertex set  $\{\alpha_1+\alpha_2,\alpha_1,\alpha_2,\alpha_3\}$ is not complete. Therefore $G(S)$ is not a block graph. We get a contradiction to the necessity of Theorem~\ref{thm1.1}.
\end{exam}

\section{A modification of Theorem~\ref{thm1.1}}

In Section 2 we showed Theorem~\ref{thm1.1} to be incorrect by example. In this section we give a replacement theorem as follows.

\begin{thm}\label{thm3.1}
Assume that $S$ is a linearly independent set of $V$ and no element of $S$ is in ${\rm rad}\hspace{0.5mm}V.$
Then $Tv(S)$ is isomorphic to a symmetric group if and only if $G(S)$ is a claw-free block graph.

\end{thm}

The original proof of the sufficiency of Theorem~\ref{thm1.1} does not use the assumption
that $S$ spans $V.$ We actually get the following result.

\begin{lem}\label{lem3.3}
Let $S$ denote a subset of $V.$ 
If $G(S)$ is a claw-free block graph of order $n$ then $Tv(S)$ is isomorphic to the symmetric group on $n+1$ letters.
\end{lem}

We will not prove the necessity of Theorem~\ref{thm3.1} by revising the proof of the necessity of Theorem~\ref{thm1.1}. Instead we will provide a short proof. To do this we define two binary relations on the set of all linearly independent sets of $V$ and need three lemmas.

\begin{defn}\label{defn3.1}
Define $\mathcal{I}_0$ to be the restriction of $\mathcal{T}_0$ to the set of all linearly independent sets of $V.$
\end{defn}

Observe that the binary relation $\mathcal I_0$ from Definition~\ref{defn3.1} is symmetric.

\begin{defn}\label{defn3.2}
Define $\mathcal{I}$ to be the equivalence relation on the set of all linearly independent sets of $V$ generated by $\mathcal{I}_0.$
\end{defn}


The original proof of Lemma~\ref{lem2.1} works for the following lemma.

\begin{lem}\label{lem3.1}
Assume that $S$ is a linearly independent set of $V$ and no element of $S$ is in ${\rm rad}\hspace{0.5mm}V.$ Let the equivalence relation $\mathcal I$ be as in Definition~\ref{defn3.2}. Then $Tv(S)$ is isomorphic to a symmetric group if and only if there exists $S'$ in the $\mathcal I$-equivalence class of $S$ for which $G(S')$ is a path.
\end{lem}

To state the second lemma we recall the notion of the line graph of a simple graph.
Let $G$ denote a simple graph. The {\it line graph of $G$} is a
simple graph that has a vertex for each edge of $G,$ and two of
these vertices are adjacent whenever the corresponding edges in
$G$ have a common vertex.

\begin{lem}\label{lem3.2}
{\rm \cite[Theorem 8.5]{ha:72}.}
Let $G$ denote a simple graph. Then $G$ is a claw-free block graph if and only if $G$ is the line graph of a tree.
\end{lem}

\begin{lem}\label{lem3.4}
Let $S$ denote a linearly independent set of $V$ for which $G(S)$ is a claw-free block graph. Then for each $S'$ in the $\mathcal I$-equivalence class of $S$ the graph $G(S')$ is a claw-free block graph.
\end{lem}
\begin{proof}
Let $S'$ denote a subset of $V$ which $S$ is $\mathcal I_0$-related to. Let $\alpha,\beta\in S$ such that $S'$ is obtained from $S$ by changing $\beta$ to $\tau_\alpha \beta.$ If $B(\alpha,\beta)=0$ there is nothing to prove. Thus we assume $B(\alpha,\beta)=1.$ By Lemma~\ref{lem3.2} there exists a tree $T$ whose line graph is $G(S).$ Let $u$ denote the common vertex of the edges $\alpha$ and $\beta$ in $T.$ Let $v$ and $w$ denote the other vertices incident to $\alpha$ and $\beta$ in $T,$ respectively. Let $T'$ denote the tree obtained from $T$ by removing the edge $\beta$ and adding a new edge between $v$ and $w.$ We call the new edge $\tau_\alpha \beta.$ For each $\gamma\in S',$ $\gamma$ is adjacent to $\tau_\alpha \beta$ in $G(S')$ if and only if $\gamma$ is adjacent to exactly one of $\alpha$ and $\beta$ in $G(S).$ Therefore $G(S')$ is the line graph of $T'$ so $G(S')$ is a claw-free block graph by Lemma~\ref{lem3.2}. The result follows since $\mathcal I_0$ is symmetric and generates $\mathcal I.$
\end{proof}

\noindent{\it Proof of Theorem~\ref{thm3.1}.}
(sufficiency): Immediate from Lemma~\ref{lem3.3}.

\noindent (necessity): By Lemma~\ref{lem3.1} there exists $S'$ in the $\mathcal I$-equivalence class of $S$ for which $G(S')$ is a path. Since a path is a claw-free block graph and by Lemma~\ref{lem3.4} the graph $G(S)$ is a claw-free block graph. \hfill $\Box$

\begin{cor}\label{cor3.1}
Assume that $S$ is a linearly independent set of $V$ and no element of $S$ is in ${\rm rad}\hspace{0.5mm}V.$ Let the equivalence relation $\mathcal I$ be as in Definition~\ref{defn3.2}. Then the following {\rm (i)}--{\rm (iv)} are equivalent:
\begin{enumerate}
\item[{\rm (i)}] $Tv(S)$ is isomorphic to a symmetric group.

\item[{\rm (ii)}] $G(S)$ is a claw-free block graph.

\item[{\rm (iii)}] $G(S)$ is the line graph of a tree.

\item[{\rm (iv)}] There exists $S'$ in the $\mathcal I$-equivalence class of $S$ for which $G(S')$ is a path.
\end{enumerate}
Suppose {\rm(i)}--{\rm (iv)} hold. Let $n$ denote the cardinality of $S.$ Then $Tv(S)$ is isomorphic to the symmetric group on $n+1$ letters.
\end{cor}
\begin{proof}
(i) $\Leftrightarrow$ (ii): Immediate from Theorem~\ref{thm3.1}.

(i) $\Leftrightarrow$ (iv): Immediate from Lemma~\ref{lem3.1}.

(ii) $\Leftrightarrow$ (iii): Immediate from Lemma~\ref{lem3.2}.

The last assertion is immediate from Lemma~\ref{lem3.3}.
\end{proof}

\section{Comments}

Given Theorem \ref{thm3.1} it is natural to further study the linearly dependent sets $S$ of $V$ for which the equivalence holds. This section is devoted to a description of these linearly dependent sets.

In view of Lemma~\ref{lem3.3} it is enough to study the linearly dependent set $S$ of $V$ for which $G(S)$ is a claw-free block graph (equivalently, the line graph of a tree). Moreover, replacing $V$ by the subspace of $V$ spanned by $S$ if necessary, we may assume without loss of generality that $S$ spans $V.$


We now describe how to obtain such a set $S$. For convenience an edge of a tree incident to a leaf will be said to be a {\it pendant edge}. Assume that $V$ has zero radical. Let $I$ denote a basis of $V$ for which $G(I)$ is the line graph of a tree $T.$ Pick a vertex $u$ of $T.$ Since the radical of $V$ is zero there exists a unique $\beta\in V$ such that for each $\alpha\in I$,
\begin{equation}\label{e4.1}
B(\alpha,\beta)=
\left\{
\begin{array}{ll}
1 \qquad &\hbox{if $u$ is incident to $\alpha$ in $T$},\\
0 \qquad &\hbox{if $u$ is not incident to $\alpha$ in $T$.}
\end{array}
\right.
\tag{\dag}
\end{equation}
Let $S=I\cup\{\beta\},$ which is linearly dependent unless the dimension of $V$ is two and $u$ is a leaf of $T.$ Suppose that $S$ is linearly dependent. Let $\it \Upsilon$ denote the tree obtained from $T$ by adding a pendant edge incident to $u.$ We call the new edge $\beta.$ Then $G(S)$ is the line graph of $\it \Upsilon.$

At the end of this section we will see that any linearly dependent spanning set $S$ of $V$ for which $G(S)$ is the line graph of a tree can be obtained in the above way. To this end we establish two lemmas.

\begin{lem}\label{lem4.2}
Assume that $V$ has zero radical. Let $I$ denote a basis of $V$ for which $G(I)$ is connected. Then for any $k\geq 2$ mutually distinct vectors $\beta_1,\ldots,\beta_k\in V\setminus I,$ the graph  $G(I\cup\{\beta_1,\ldots,\beta_k\})$ is not the line graph of a tree.
\end{lem}
\begin{proof}
Proceed by contradiction. Suppose there exist distinct $\beta,\gamma\in V\setminus I$ such that $G(I\cup\{\beta,\gamma\})$ is the line graph of a tree $\it \Upsilon.$ Let $T$ denote the subgraph of $\it \Upsilon$ induced by all $\alpha\in I.$ Since $I$ spans $V$ each of $\beta$ and $\gamma$ is a pendant edge of $\it \Upsilon.$ Let $u$ and $v$ denote the two vertices of $T$ incident to $\beta$ and $\gamma$, respectively. Since the radical of $V$ is zero $u$ and $v$ are distinct. Let $\alpha_1,\ldots,\alpha_k$ denote the edges in the path joining $u$ and $v.$ The incidence relation on $\it \Upsilon$ implies that
$
B(\alpha,\beta+\gamma)=B(\alpha,\alpha_1+\cdots+\alpha_k)
$
for each $\alpha\in I.$
Therefore $\beta+\gamma=\alpha_1+\cdots+\alpha_k.$ Using this we deduce $B(\beta,\gamma)=1,$ a contradiction.
\end{proof}


\begin{lem}\label{lem4.1}
Assume that $I$ is a basis of $V$ for which $G(I)$ is the line graph of a tree $T.$ Then the following {\rm (i)}--{\rm (iv)} are equivalent:
\begin{enumerate}
\item[{\rm (i)}] The radical of $V$ is zero.

\item[{\rm (ii)}] The dimension of $V$ is even.

\item[{\rm (iii)}] For some vertex $u$ in $T$ there exists $\beta\in V$ such that {\rm (\ref{e4.1})} holds for each $\alpha\in I.$

\item[{\rm (iv)}] For each vertex $u$ in $T$ there exists a unique $\beta\in V$ such that {\rm (\ref{e4.1})} holds for each $\alpha\in I.$
\end{enumerate}
\end{lem}
\begin{proof}
Let $V^\ast$ denote the dual space of $V.$ Define a linear map $\theta:V\to V^\ast$ by
\begin{eqnarray*}
\theta(\alpha)\beta=B(\alpha,\beta) \qquad \hbox{for all $\alpha,\beta\in V.$}
\end{eqnarray*}
The kernel of $\theta$ is ${\rm rad}\hspace{0.5mm}V.$ Therefore (i) if and only if (i$'$) the map $\theta$ is a bijection. We show that (i$'$) and (ii)--(iv) are equivalent. Condition (i$'$) immediately implies (iv). To see that (iv) implies (i$'$) we let $U$ denote the vertex space of $T$ over $\mathbb F_2$ and define a linear map $\lambda:U\to V^\ast$ by for all vertices $u$ of $T$ and for all $\alpha\in I,$
\begin{eqnarray*}
\lambda(u)\alpha=\left\{
\begin{array}{ll}
1 \qquad &\hbox{if $u$ is incident to $\alpha$ in $T$,}\\
0 \qquad &\hbox{if $u$ is not incident to $\alpha$ in $T$.}\\
\end{array}
\right.
\end{eqnarray*}
The kernel of $\lambda$ is $\{0,w\},$ where $w$ is the sum of all vertices of $T$. By dimension theorem $\lambda$ is surjective. Therefore (iv) implies that $\theta$ is surjective and so is bijective.

To see the equivalence of (ii)--(iv) we define a linear map $\mu:V\to U$ by for all $\alpha\in I,$
$$
\mu(\alpha)=u+v,
$$
where $u$ and $v$ are the two distinct vertices incident to $\alpha$ in $T.$ Observe that $\theta=\lambda\circ \mu$ and that the image of $\mu,$ denoted by ${\rm Im}\hspace{0.5mm}\mu,$ consists of all $v\in U$ each of which is equal to the sum of an even number of vertices in $T$. Therefore (ii) if and only if $w+u\in {\rm Im}\hspace{0.5mm}\mu$ for some (resp. each) vertex $u$ of $T$ if and only if (iii) (resp. (iv)). Here $w$ is the nonzero vector in the kernel of $\lambda.$
\end{proof}

\begin{prop}
Let $S$ denote a linearly dependent spanning set of $V.$ Assume that $G(S)$ is the line graph of a tree $\it \Upsilon.$ Then the following {\rm (i)}--{\rm (iii)} hold:
\begin{enumerate}
\item[{\rm (i)}] The radical of $V$ is zero.

\item[{\rm (ii)}] The dimension of $V$ is even.

\item[{\rm (iii)}] For each pendant edge $\beta$ of $\it \Upsilon$ the set $S\setminus \{\beta\}$ is a basis of $V.$
\end{enumerate}
\end{prop}
\begin{proof}
Consider the set consisting of the linearly independent subsets $I$ of $S$ for which $G(I)$ is connected. From this set we choose a maximal element $J$ under inclusion. The maximality of $J$ forces that each $\alpha\in S\setminus J$ is in the subspace of $V$ spanned by $J.$ Therefore $J$ is a basis of $V.$ Applying Lemma~\ref{lem4.1} to $I=J,$ (i) and (ii) follow. To prove (iii) we fix a pendant edge $\beta$ of $\it \Upsilon$ and show that $S\setminus \{\beta\}$ is linearly independent. Applying Lemma~\ref{lem4.2} to $I=J$ it follows that $S\setminus J$ contains exactly one element, denoted by $\alpha.$ If $\alpha=\beta$ there is nothing to prove. Thus we assume  $\alpha\neq\beta.$ Let $W$ denote the subspace of $V$ spanned by $J\setminus\{\beta\},$ which has odd dimension. Applying Lemma~\ref{lem4.1} to $I'=J\setminus\{\beta\}$ and $V'=W$ we find that $\alpha\not\in W.$ Therefore $S\setminus\{\beta\}$ is linearly independent.
\end{proof}

\section{Acknowledgements}
The author would like to thank the anonymous referee for valuable suggestions.


\bigskip

\noindent Hau-wen Huang
\hfil\break Department of Applied Mathematics
\hfil\break National Chiao Tung University
\hfil\break 1001 Ta Hsueh Road
\hfil\break Hsinchu, Taiwan, 30050.
\hfil\break Email:  {\tt poker80.am94g@nctu.edu.tw}
\medskip

\end{document}